\documentclass[reqno,11pt]{amsart}
\usepackage[applemac]{inputenc}
\usepackage{amsmath}
\usepackage{enumerate}
\usepackage{amssymb}
\usepackage{mathrsfs}
\usepackage{a4wide}
\usepackage{amsthm}
\usepackage{epsfig}
\usepackage{color}
\numberwithin{equation}{section}

\newtheorem{theorem}{Theorem}[section]

\newtheorem{definition}[theorem]{Definition}
\newtheorem{pro}[theorem]{Proposition}
\newtheorem{corollary}[theorem]{Corollary}

\def\r2{\mathbb{R}^2}

\def\rd{\mathbb{R}^d}

\def\div{\mathrm{div}\,}

\def\id{\mbox{\boldmath$\mathfrak{i}$}}

\newcommand{\BBB}{\color{black}}

\newcommand{\R}{\mathbb{R}}

\newcommand{\FF}{\mathscr{F}}
\newcommand{\GG}{\mathscr{G}}

\newcommand{\tauV}{{\kern-3pt\tau}}

\newcommand{\la}{{\langle}}                  
\newcommand{\ra}{{\rangle}}
\newcommand{\eps}{\varepsilon}

\newcommand{\restr}[1]{\lower3pt\hbox{$|_{#1}$}}

\newcommand{\Pc}[2]{\overline{#1}\kern-2pt^{\vphantom 0}_{#2}}
\newcommand{\Pcb}[2]{\underline{\phantom{..}}\kern-6pt #1_{#2}}

\newcommand{\nchi}{{\raise.3ex\hbox{$\chi$}}}
\def\la{\langle}
\def\ra{\rangle}

\def\r2{\mathbb{R}^2}

\def\div{\mathrm{div}\,}

\def\dt0{{{\frac{d}{dt}}_{|t=0}}}
\def\({\left(}
\def\){\right)}

\def\id{\mathbf{1}}

\def\rd{\mathbb{R}^d}

\def\id{\mbox{\boldmath$\mathfrak{i}$}}

\def \eds#1#2#3 {#1, #2, #3.}
\def \no#1#2#3 {{\bf #1} (#2), #3.}
\def \nono#1#2#3#4 {{\bf#1} (#2), no.#3, #4.}

\newcommand{\mass}{\mathfrak{m}}
\newcommand{\MeasuresTwo}{\mathscr M_2(\rd;\mass)}

\newenvironment{proofad}{\removelastskip\par\medskip   
\noindent{\em Proof of Theorems \ref{main1} and \ref{main2}.}
\rm}{\penalty-20\null\hfill$\square$\par\medbreak} 

\newenvironment{proofad1}{\removelastskip\par\medskip   
\noindent{\em Proof of Proposition \ref{pro1}.}
\rm}{\penalty-20\null\hfill$\square$\par\medbreak} 

\newenvironment{proofad2}{\removelastskip\par\medskip   
\noindent{\em Proof of Proposition \ref{pro2}.}
\rm}{\penalty-20\null\hfill$\square$\par\medbreak} 

\title[Uniqueness for Keller-Segel models]{Uniqueness for Keller-Segel-type chemotaxis models}

\author{J. A. Carrillo, S. Lisini, E. Mainini}

\address{Jos\'e A. Carrillo, Department of Mathematics, Imperial College London,
South Kensington Campus, London SW7 2AZ, UK.}
\email{carrillo@imperial.ac.uk}

\address{Stefano Lisini,
Dipartimento di Matematica ``F. Casorati'', Universit\`a degli
Studi di Pavia, via Ferrata 1, 27100 Pavia, Italy.}
\email{stefano.lisini@unipv.it}

 \address{Edoardo Mainini,
Universit\`a degli Studi di Genova, Dipartimento di Ingegneria meccanica, energetica, gestionale e dei trasporti (DIME) - sezione MAT. P.le Kennedy 1, 16129 Genova, Italy}
\email{edoardo.mainini@unipv.it}

\date{}

\begin{document}

\begin{abstract}
We prove uniqueness in the class of integrable and bounded
nonnegative solutions in the energy sense to the Keller-Segel (KS)
chemotaxis system. Our proof works for the fully parabolic KS
model, it includes the classical parabolic-elliptic KS equation as
a particular case, and it can be generalized to nonlinear
diffusions in the particle density equation as long as the
diffusion satisfies the classical McCann displacement convexity
condition. The strategy uses Quasi-Lipschitz estimates for the
chemoattractant equation and the above-the-tangent
characterizations of displacement convexity. As a consequence, the
displacement convexity of the free energy functional associated to
the KS system is obtained from its evolution for bounded
integrable initial data.
\end{abstract}

\maketitle

\let\thefootnote\relax



\section{Introduction}

The classical Keller-Segel (KS) model for chemotaxis is the system
\[
\left\{\begin{array}{rl}
\partial_t n&=\kappa\Delta n-\chi\, \div(n\nabla c),\\
\partial_t c&=\eta\Delta c +\theta n-\gamma c.
\end{array}\right.
\]
Here, $n$ is the number/mass density of a bacteria/cell population
and $c$ represents  the concentration of a chemical attractant
that can suffer chemical degradation and that it is produced by
the cells themselves due to chemotactic interaction. The
parameters $\kappa,\chi,\eta, \theta,\gamma$ might be suitable
functions, assumed to be constant in this simplified model. We can
perform a time scaling and a suitable change of variables, that is
$\tau=\kappa t$,
$\rho(x,\tau)=\frac{\theta\chi}{\eta\kappa}n(x,\tau/\kappa),
v(x,\tau)=\frac\chi\kappa c(x,\tau/\kappa)$.
 The system is therefore reduced to
 \begin{equation}\label{ks}
\left\{\begin{array}{rl}
\partial_t \rho  &=\Delta \rho-\, \div(\rho\nabla v),\\[2mm]
\eps\partial_t v &=\Delta v + \rho-\alpha v,
\end{array}\right.
\end{equation}
where $\alpha\ge 0 $ and $\eps\ge 0$ are constants
($\alpha=\gamma/\eta,\: \eps=\kappa/\eta$). In case $\eps=0$, it
restricts to the classical parabolic-elliptic Patlak-KS model
\begin{equation}\label{pe}
\left\{\begin{array}{l}
\partial_t \rho=\Delta \rho-\, \div(\rho\nabla v),\\[2mm]
-\Delta v + \alpha v= \rho.
\end{array}\right.
\end{equation}
For $\eps>0$, the natural free energy functional associated to the
dynamics of the system \eqref{ks} is
\begin{equation}\label{fullfunctional}
\FF_{\eps,\alpha}(\rho,v):=\int_{\mathbb{R}^d}(\rho\log\rho-v\rho)\,
dx +\frac12\int_{\mathbb{R}^d}(|\nabla v|^2+\alpha v^2)\, dx\,.
\end{equation}
In the case $\eps=0$, corresponding to \eqref{pe}, this Liapunov
functional is at least formally equivalent to
\begin{equation}\label{functional2}
\FF_{0,\alpha}(\rho):=\int_{\mathbb{R}^d}(\rho\log\rho-\frac12
v\rho)\, dx\,
\end{equation}
 with the  convention that $v$ is obtained from the density $\rho$ by
$v=\mathcal{B}_{\alpha,d}\ast \rho$. Here,
$\mathcal{B}_{\alpha,d}$ denotes the Bessel kernel for $\alpha>0$
or the Newtonian kernel for $\alpha=0$, for any dimension $d$.
Therefore the role of the parameter $\eps$ is to discriminate
between parabolic-parabolic and parabolic-elliptic system. Note
that the Liapunov functionals \eqref{fullfunctional} and
\eqref{functional2} are just formally equivalent since the
$L^2$-integrability of $\nabla \mathcal{B}_{\alpha,d}\ast\rho$
fails if $d=1,2$ and $\alpha=0$. Thus, even if the classical free
energy writing and valid for all cases when $\eps=0$ is the one in
\eqref{functional2}, we will prefer to work with the functional as
in \eqref{fullfunctional} even if $\eps=0$, with a suitable
renormalization for the cases $d=1,2$ and $\alpha=0$ discussed in
Section 3.

Our main objective is the uniqueness of certain solutions, for
both systems \eqref{ks} and \eqref{pe}. Let us introduce the
notion of solution for the Cauchy problems associated to
\eqref{ks} and \eqref{pe} that we will consider in this work. We
denote by $\MeasuresTwo$ the set of nonnegative densities over
$\mathbb{R}^d$ with mass $\mass$ and finite second moment, i.e.,
$$
\MeasuresTwo:=\left\{\rho\in L^1(\rd): \rho\geq 0,
\int_{\rd}\rho(x)\,dx = \mass, \int_{\rd}|x|^2\rho(x)\,dx<+\infty
\right\}.
$$

\begin{definition}\label{def1}
We say that a weakly continuous map $\rho\in
C_w([0,T];\MeasuresTwo)$ is a bounded solution to the Cauchy
problem for \eqref{pe}, with initial datum
$\rho^0\in\MeasuresTwo\cap L^\infty(\mathbb{R}^d)$, if
\begin{itemize}
\item[\it i)] $\rho\in L^∞((0,T)\times \mathbb{R}^d)$ and
$|x|^2\rho_t(x)\in L^∞((0,T),L^1(\mathbb{R}^d))$,

\item[\it ii)] $\rho_0=\rho^0$ and the first equation of
\eqref{pe} holds in the sense of distributions on $(0,T)\times
\mathbb{R}^d$, where $v_t=\mathcal{B}_{\alpha,d}\ast \rho_t$ for
all $t\in [0,T]$,

\item[\it iii)] $\rho_t\in W^{1,1}(\mathbb{R}^d)$ for
$\mathcal{L}^1$-a.e. $t\in (0,T)$ and
\begin{equation}\label{dissipation}
\int_0^T\int_{\mathbb{R}^d}\frac{|\nabla\rho_t(x)|^2}{\rho_t(x)}\,dx\,dt
< +\infty.
\end{equation}
\end{itemize}
\end{definition}

\begin{definition}\label{def2}
We say that the couple of functions $ (\rho,v)$, with $\rho\in
C_w([0,T];\MeasuresTwo)$ and $v\in L^2((0,T);
W^{1,\,2}(\mathbb{R}^d))$, is a bounded solution to \eqref{ks}
with initial datum $(\rho^0,v^0)$, $\rho^0\in\MeasuresTwo\cap
L^\infty(\mathbb{R}^d)$ and $v^0\in W^{1,2}(\mathbb{R}^d)$, if
\begin{itemize}
\item[\it I)] $\rho\in L^∞((0,T)\times \mathbb{R}^d)$ and
$|x|^2\rho_t(x)\in L^∞((0,T),L^1(\mathbb{R}^d))$,

\item[\it II)] $\rho_0=\rho^0$,  the first
equation of \eqref{ks} holds in the sense of distributions on
$(0,T)\times\mathbb{R}^d$, and $v$ is the unique solution to the
Cauchy problem for the forced parabolic equation $\eps\partial_t
v-\Delta v+\alpha v=\rho$ over $(0,T)\times\mathbb{R}^d$ in the
standard sense, with initial datum $v^0$,

\item[\it III)] the property {\it iii)} of  {\rm Definition
\ref{def1}} holds.
\end{itemize}
\end{definition}

Let us emphasize that the main properties we need to get
uniqueness of solution are the boundedness of the densities and
the Fisher information \eqref{dissipation}. They together imply that the velocity
field of the continuity equation for the density $\rho$ is a well
defined object belonging to the right functional space, see
section 2 for details. Moreover, the boundedness of the density
implies that we have a uniform in bounded time intervals estimate
on the quasi-Lipschitz constant of part of the velocity field.
These are the basic properties that imply the uniqueness for
bounded solutions. Let us finally mention that part of the
strategy is related to the uniqueness of solutions to fluid and
aggregation equations developed in \cite{Y,L,AS,M1,CR,BLR,M2}. The
main novelty here is the interplay between the diffusive and the
aggregation parts. The main results of this work are:

\begin{theorem}\label{main1}
Let $T>0$ and let $\rho^0\in \MeasuresTwo\cap
L^\infty(\mathbb{R}^d)$. Let $\rho_1, \rho_2$ be two bounded
solutions on $[0,T]\times\mathbb{R}^d$ to the Cauchy problem
associated to \eqref{pe}, with initial datum $\rho^0$. Then
$\rho_1=\rho_2$.
\end{theorem}

\begin{theorem}\label{main2}
Let $T>0$ and let $\rho^0\in \MeasuresTwo\cap
L^\infty(\mathbb{R}^d)$, $v^0\in W^{1,2}(\mathbb{R}^d) \cap
W^{2,\infty}(\mathbb{R}^d)$. Let $(\rho_1, v_1)$ and
$(\rho_2,v_2)$ be two bounded solutions on
$[0,T]\times\mathbb{R}^d$ of the Cauchy problem associated to
\eqref{ks}, with initial datum $(\rho^0,v^0)$. Then $(\rho_1,
v_1)=(\rho_2,v_2)$.
\end{theorem}

The proof of  uniqueness  as stated in Theorems \ref{main1} and
\ref{main2} will be a consequence of a more general property: we
will show that bounded solutions satisfies a strong gradient flow
formulation by means of a family of evolution variational
inequalities. This formulation is similar to the one for
semi-convex functionals and implies a non-expansivity property of
the distance between two solutions. This non-expansivity property
yields uniqueness. All these results will be stated in Theorem
\ref{th:GFEVI}. Moreover this formulation lead also to a relaxed
convexity property of the energy functional as stated in Theorem
\ref{th:convexity}.

There is a huge literature about the KS system and their
variations, so we just restrict here to discuss the main results
concerning about bounded solutions. In the classical
parabolic-elliptic KS equation $\eps=\alpha=0$ and $d=2$, global
in time bounded solutions in the subcritical case $\mass< 8\pi$
have been obtained joining the results in \cite{BDP,Kow05,CC}.
Actually, the global existence of weak solutions satisfying all
properties in Definition \ref{def1} except the $L^\infty$ bound
was obtained in \cite{BDP} while $L^\infty$-bounds in bounded time
intervals can be obtained from the results in \cite{Kow05,CC}. The
same techniques could eventually be used to get local in time
bounded solutions for all masses, although such a result is not
present in the literature. Let us also mention the recent preprint
\cite{CD} in which the authors actually show that the
$L^\infty$-norm of the solution decays in time like for the heat
equation in the subcritical case $\mass< 8\pi$ for more restricted
initial data. $L^\infty$-apriori estimates were obtained in the
classical parabolic-elliptic KS equation $\eps=0$ with $d\geq 2$
and $\alpha\geq 0$ for small $L^{d/2}$ initial data in
\cite{CPZ04,CPZ}. These results together with similar arguments as
in \cite{BDP} to get the free energy dissipation property and thus
the Fisher information bounds, could lead to the existence of
bounded solutions in these cases. We emphasize that these
$L^\infty$ estimates show that the solution in bounded time
intervals is bounded by a constant that depends only on the
$L^\infty$-norm of the initial data, the initial free energy, and
the final time. In particular, existence of bounded solutions is
expected if $\rho^0\in L^\infty(\mathbb{R}^d)$, and this explains
the presence of such an assumption in the previous definitions.

Concerning the fully parabolic KS system, we find global in time
solutions satisfying all properties stated in Definition
\ref{def2} except the $L^\infty$ bounds in \cite{CaCo} for $d=2$
and the subcritical mass case $\mass < 8\pi$. $L^\infty$-apriori
estimates were obtained in \cite{KS} for the fully parabolic case
but in bounded domains. It is reasonable to expect that this
strategy should work for the whole space case, although it is not
written as such in the literature. Results in higher dimensions
concerning solutions with $L^\infty$ estimates for small initial
data can be found in \cite{Bil99} but estimates on the free energy
dissipation are missing there. We finally refer to
\cite{BCL,BL,CC,KimYao} for different results concerning the
existence of solutions satisfying the boundedness of the Fisher
information and/or the uniform bounds of the solutions for
particular choices of $\eps\geq 0$, $\alpha\geq 0$, and nonlinear
diffusions.

As mentioned before, Theorems \ref{main1} and \ref{main2} are
based on the derivation of quasi-Lipschitz estimates for the
chemoattractant $v$. This is the reason behind the additional
assumption on the initial datum $v^0$ in Theorem \ref{main2}. We
will clarify the use of quasi-Lipschitz estimates of the
chemoattractant in the next section together with a quick summary
of the main properties of optimal transport that we need in this
work. Section \ref{sec:main} is devoted to show the main
uniqueness results, derived from a more general property of
bounded solutions for the Keller-Segel model. In fact, we will
show that for bounded solutions we can obtain evolution
variational inequalities. In Section \ref{sec:convexity} we show
that these evolution variational inequalities lead to certain
convexity of the associated free energy functional. In order not
to break the flow of the argument, we postpone to Section
\ref{sec:zygmund} the rigorous derivation of the quasi-Lipschitz
estimates of the elliptic and parabolic equations for $v$. In
Section \ref{sec:zygmund}, we will also prove a strengthening of
Theorem \ref{main2}, with more general initial data. Finally,
Section \ref{sec:nonlinear} is devoted to show how to adapt these
arguments to Keller-Segel models with nonlinear diffusion.


\section{Preliminary notions}\label{prel}

\subsection{Some elliptic and parabolic regularity estimates}

The proofs of our results are based on the technique used by
Yudovich \cite{Y} for treating uniqueness in the case of
incompressible Euler equations for fluidodynamics. In particular,
we exploit a quasi-Lipschitz property for the velocity field of
the continuity equation for $\rho$ in \eqref{ks} and \eqref{pe}.
This property comes from the regularity that $v$ gains being
solution to the second equation in \eqref{ks} and \eqref{pe}.

Suppose first that $v=\mathcal{B}_{0,d}\ast\rho$. If $\rho\in
L^1\cap L^\infty(\mathbb{R}^d)$, by exploiting some estimates of
the Newtonian potential, $\nabla v$ satisfies the following
log-Lipschitz property (see \cite{BLL} and \cite[Chapter 8]{MB},
\cite{SV} and also \cite{Y}),
\begin{equation*}
|\nabla v(x)-\nabla v(y)|\le C|x-y|(1-\log^-|x-y|),
\end{equation*}
where $C$ is a suitable positive constant, depending only on
$\|\rho\|_{L^1}$ and $\|\rho\|_{L^\infty}$ and  $\log^-$ denotes
the negative part of the natural logarithm function. As a
consequence, we get the estimate
\begin{equation}\label{loglipschitz}
|\nabla v(x)-\nabla v(y)|^2\le C^2\varphi(|x-y|^2)
\end{equation}
for some new positive constant $C$, where  $\varphi$ is the
concave function on $[0,\infty)$ defined as
\begin{equation}\label{varphi}
\varphi(x):=\left\{
\begin{array}{ll}
x\log^2 x\quad&\mbox{if $x\le e^{-1-\sqrt{2}}$},\\
x+2(1+\sqrt2)e^{-1-\sqrt2}\quad&\mbox{if $x> e^{-1-\sqrt{2}}$}.
\end{array}
\right.
\end{equation}
Indeed, for large values of $|x-y|$ the estimate
\eqref{loglipschitz} is quite obvious, since it is immediate to
show that $\nabla \mathcal{B}_{0,d}\ast \rho$ is a bounded
function in the whole space with a direct estimate using the fact
that $\rho\in L^1\cap L^\infty(\mathbb{R}^d)$. The log-Lipschitz
property itself can be justified through standard elliptic
regularity, as we will do in Section \ref{sec:zygmund}.

Analogous facts hold if we consider the equation $-\Delta v+\alpha
v=\rho$, appearing in \eqref{pe}, or more general uniformly
elliptic operators, so that we have

\begin{pro}\label{pro1}
Suppose that $\rho\in L^1\cap L^\infty({\mathbb{R}^d})$ and
$\alpha\ge 0$. Then $v=\mathcal{B}_{\alpha,d}\ast\rho$ satisfies
the estimate \eqref{loglipschitz}, where $C$ is a suitable
positive constant, depending only on $\alpha, d, \|\rho\|_{L^1(\mathbb{R}^d)}$,
and $\|\rho\|_{L^\infty(\mathbb{R}^d)}$.
\end{pro}

About the parabolic equation for $v$ in \eqref{ks}, the
quasi-Lipschitz property also carries over, since formally
inequality \eqref{loglipschitz} translates in terms of the
parabolic metric to
\begin{equation}\label{parabolicloglipschitz}
|\nabla v(t,x)-\nabla v(s,y)|^2\le C^2\varphi((|x-y|+|s-t|^{1/2})^2)\quad\forall x,y\in\mathbb{R}^d, \; s,t\in[0,T].
\end{equation}

\begin{pro}\label{pro2}
Suppose that $\rho\in L^\infty((0,T)\times{\mathbb{R}^d})$, $v^0
\in W^{2,\infty}(\mathbb{R}^d)$ and $\alpha\ge 0$. If $v$ is the
unique solution to the Cauchy problem for the parabolic equation
$\partial_t v=\Delta v-\alpha v+\rho$ (in the standard sense of
convolution with fundamental solution), then $v$ satisfies
\eqref{parabolicloglipschitz}, where $C$ is a suitable positive
constant, depending only on $\alpha, d,
\|v^0\|_{W^{2,\infty}(\mathbb{R}^d)}$, and
$\|\rho\|_{L^\infty((0,T)\times{\mathbb{R}^d})}$.
\end{pro}

For a more complete insight into these properties, it will be
convenient to recall some facts about the Zygmund class and its
role in elliptic and parabolic regularity. However, in order not
to introduce some not really necessary notation before the proof
of our main results, we prefer to postpone the proof of
Proposition \ref{pro1} and Proposition \ref{pro2} to Section
\ref{sec:zygmund}. Indeed, in Section \ref{sec:zygmund} we will develop a
more rigorous discussion about the log-Lipschitz estimates, and
thanks to some refined parabolic regularity we will also prove a
slight strengthening of Theorem \ref{main2}.


\subsection{Elementary notions of optimal transport}
Given $\rho_0,\rho_1 \in \MeasuresTwo$,
we define the Wasserstein distance between $\rho_0$ and $\rho_1$ as
\[
W_2(\rho_0,\rho_1)=\left(\int_{\mathbb{R}^d}|x-\mathcal{T}(x)|^2\,\rho_0(x)\,dx\right)^{\frac12},
\]
where $\mathcal{T}$ is the unique optimal transport map between
$\rho_0$ and $\rho_1$, that is, the map
$\mathcal{T}:\mathbb{R}^d\to \mathbb{R}^d$ which minimizes
$\int_{\mathbb{R}^d}|x-\mathcal{S}(x)|^2\,\rho_0(x)\,dx$ among all
the Borel maps $\mathcal{S}:\mathbb{R}^d\to \mathbb{R}^d$ satisfying
$\mathcal{S}_\#\rho_0=\rho_1$. We recall that $\mathcal{S}_\#\rho_0=\rho_1$ means that
$\int_{\rd} \varphi(x)\rho_1(x)\,dx = \int_{\rd}
\varphi(\mathcal{S}(x))\rho_0(x)\,dx$ for every continuous and bounded
function $\varphi:\rd\to\rd$.

The Wasserstein geodesic between $\rho_0$ and $\rho_1$ is the
curve $s\in[0,1]\mapsto \rho^s\in \MeasuresTwo$ defined by the
so-called displacement interpolation along the optimal transport
map $\mathcal{T}$ between $\rho_0$ and $\rho_1$, that is,
$\rho^s:=((1-s)\id+s\mathcal{T})_\#\rho_0$. In particular, for any
$s$, $\mathcal{T}_s:= (1-s)\id+s\mathcal{T}$ is the optimal map
between $\rho_0$ and $\rho^s$ and there holds
$W_2(\rho^r,\rho^s)=|s-r|W_2(\rho_0,\rho_1)$.

We recall a formula for the differentiation of the squared Wasserstein
distance  along solutions of the continuity equation. Let
$t\in [0,T]\mapsto \rho_t\in \MeasuresTwo$ be a weakly continuous
curve which is distributional solution of
\begin{equation*}
\partial_t\rho_t+\div(\xi_t\rho_t)=0,
\end{equation*}
for some Borel velocity field $\xi_t$ such that
$\int_0^T\|\xi_t\|^2_{L^2(\mathbb{R}^d,\rho_t;\mathbb{R}^d)}\,dt<+\infty$.
Then the curve is absolutely continuous with respect to the
Wasserstein distance, \cite[Theorem 8.3.1]{AGS}. Then, for any
$\bar\rho\in\MeasuresTwo$, it holds
\begin{equation}\label{timederivative}
\frac12\frac{d}{dt} W_2^2(\rho_t,\bar\rho)=\int_{\mathbb{R}^d}\langle \xi_t(x),x-\mathcal{T}_t(x)\rangle\,\rho_t(x)\,dx,\qquad \mbox{for $\mathcal{L}^1$-a.e. $t\in (0,T)$},
\end{equation}
where $\mathcal{T}_t$ is the optimal map between $\rho_t$ and
$\bar\rho$ (see \cite[Theorem 8.4.7, Remark 8.4.8]{AGS}).

Finally, let us recall an estimate relating the $2$-Wasserstein
distance and the $H^{-1}$ norm proved in \cite[Proposition
2.8]{L}. Given two nonnegative densities with the same mass
$\rho_1,\rho_2\in \MeasuresTwo\cap L^\infty(\mathbb{R}^d)$, there
holds
\begin{equation}\label{h-1}
\|\rho_1-\rho_2\|_{\dot H^{-1}(\mathbb{R}^d)}\le \max\{\|\rho_1\|_{∞},\|\rho_2\|_∞\}^{1/2}W_2(\rho_1,\rho_2).
\end{equation}
Here we are letting $\dot H^1(\mathbb{R}^d)$ be the space of
Lebesgue measurable functions $v:\rd\to\R$ such that $\|\nabla
v\|_{L^2(\mathbb{R}^d)}<+\infty$, so that $\dot
H^{-1}(\mathbb{R}^d)$ is defined by duality with functions having
finite $L^2(\mathbb{R}^d)$ norm of the gradient only. By the way,
we can also consider the space
$H^1(\mathbb{R}^d)=W^{1,2}(\mathbb{R}^d)$. In fact, from the proof
in \cite[Proposition 2.8]{L} it is not difficult to see that the
same estimate holds considering the  $H^{-1}(\mathbb{R}^d)$ space
given by duality with the full norm $(\|\nabla
v\|^2_{L^2(\mathbb{R}^d)}+\|v\|^2_{L^2(\mathbb{R}^d)})^{1/2}$.


\section{Bounded solutions as gradient flows: EVI and uniqueness}\label{sec:main}

The uniqueness Theorems \ref{main1} and \ref{main2} are
consequences of a general result interpreting bounded solutions to
\eqref{ks} (resp. \eqref{pe}) as the trajectory of the {\it
gradient flow} of the functional \eqref{fullfunctional} (resp.
\eqref{functional2}) in the appropriate metric setting. We prove
that bounded solutions satisfy a family of evolution variational
inequalities (EVI). Among different notions of gradient flow in
metric sense, the EVI formulation is stronger than other
formulations and typically corresponding to a convex structure, as
in \cite[Theorem 11.2.1]{AGS} for the theory in the Wasserstein
setting.

\textbf{Notation for the energy functional}. Before giving the
proof, we introduce some uniform notation for working with the
full functional \eqref{fullfunctional} even in the
parabolic-elliptic case. Let $\rho\in\MeasuresTwo\cap
L^{\infty}(\mathbb{R}^d)$. We are considering the free energy
functional
$$
\FF_{\eps,\alpha}(\rho,v):=\int_{\mathbb{R}^d}(\rho\log\rho-v\rho)\,
dx +\frac12\int_{\mathbb{R}^d}(|\nabla v|^2+\alpha v^2)\, dx\,,
$$
defined for $v$ being any $W^{1,2}(\mathbb{R}^d)$ function if
$\eps>0$. On the other hand, if $\eps=0$ it is understood that $v$
is given by $\mathcal{B}_{\alpha,d}\ast\rho$. Therefore the
parameter $\eps$ only indicates if we are considering problem
\eqref{ks} or \eqref{pe}. In particular, this writing of the
functional as in \eqref{fullfunctional} is valid in general, even
for $\eps=0$, except for two particular cases: $\eps=\alpha=0$ and
$d=1,2$, as discussed in the introduction. In these two cases, we
need to renormalize the free energy functional. Given $\rho^*\in
\MeasuresTwo$ a smooth and compactly supported density and
$v^*=\mathcal{B}_{0,d}\ast\rho^*$, we redefine
\eqref{fullfunctional} for $\eps=\alpha=0$ and $d=1,2$ as
\begin{equation}\label{renorm}
\FF_{0,0}(\rho,v):=\int_{\mathbb{R}^d} \left[\rho\log\rho-
v(\rho-\rho^*)\right] \,dx +\frac12\int_{\mathbb{R}^d}|\nabla
(v-v^*)|^2\, dx-\frac12 \int_{\mathbb{R}^d}\rho^*v^*\,dx\,.
\end{equation}
Notice that $\nabla(v-v^*)\in L^{2}(\mathbb{R}^d)$, as
$\rho-\rho^*$ has zero mean, see \cite{AMS,SV} for more details.

In the rest of this work, when referring to the free energy
functional $\FF_{\eps,\alpha}(\rho,v)$, we will be using
\eqref{fullfunctional} for any $\eps\ge 0, \alpha\ge 0$, except
for $\eps=\alpha=0$ and $d=1,2$ where the free energy functional
is given by \eqref{renorm}.

Let us  observe that now all the integrals involved in the definition of $\FF_{\eps,\alpha}$
are well defined and finite for $\eps\geq 0, \alpha \ge 0$ and $\rho,v$ as above. \BBB The negative part of the entropy term can be
classically treated by the Carleman inequality, see for instance
\cite[Lemma 2.2]{BCC} where the second moment bound on the density
is used. The boundedness of the density controls the positive
contribution of the entropy term together with the integrability
of $v\rho$ in case $\eps>0$ since $v\in W^{1,2}(\R^d)$.
For $\eps=0$ the integrability of $v\rho$ in case $\alpha> 0$ is
implied by the Newtonian potential case $\alpha=0$ since the
singularity of the Bessel potential at the origin is the same. The
integrability for $\alpha=\eps=0$ and $d\geq 3$ results directly
from the Hardy-Littlewood-Sobolev inequality for the Newtonian
potential. For $\alpha=\eps=0$ and $d=1,2$ we use the behavior at
infinity of the density $\rho$. Actually, $\alpha=\eps=0$ and
$d=1$ is a trivial case since the Newtonian potential is given by
$\mathcal{B}_{0,1}(x)=|x|$. For $\alpha=\eps=0$ and $d=2$ since
$\log (e+|x|^2)\rho \in L^1(\R^d)$ then $v\rho\in L^1(\R^d)$ using
the logarithmic HLS inequality, see for instance \cite{BCC2}.


\textbf{Notation for the ambient metric space}. We let
$X_\eps:=\MeasuresTwo \times L^2(\rd)$ endowed with the distance
$$
D^2(z_1,z_2)=
D^2((\rho_1,v_1),(\rho_2,v_2))=W_2^2(\rho_1,\rho_2)+\eps\|v_1-v_2\|^2_{L^2(\rd)}\,,
$$
with the convention that $X_0=\MeasuresTwo$ and
$D_0(z_1,z_2)=W_2(\rho_1,\rho_2)$. Moreover, for $z=\rho\in
X_0\times L^\infty({\mathbb{R}^d})$, $\FF_{0,\alpha}(z)$ will be
understood to be  $\FF_{0,\alpha}(\rho,v)$ with
$v=\mathcal{B}_{\alpha,d}\ast\rho$, as usual when $\eps=0$.

In the space $X_\eps$ the metric derivative of an absolutely
continuous curve $t\mapsto z_t$ is denoted and defined by
$$|z'|_{D}(t) = \lim_{h\to 0}\frac{D(z_{t+h},z_t)}{h},$$ and it
exists for $\mathcal{L}^1$-a.e. $t>0$. The local metric slope of
the functional $\FF_{\eps,\alpha}$ is defined by
$$
|\partial\FF_{\eps,\alpha}|_{D}(z):=\limsup_{D(\zeta,z)\to 0}\frac{(\FF_{\eps,\alpha}(z)-\FF_{\eps,\alpha}(\zeta))^+}{D(\zeta,z)}.
$$
These two abstractly defined objects are used to give the notion
of {\it curves of maximal slope} in general metric setting, see
\cite[§3]{AG}, \cite[Chapter 1]{AGS}. The main consequences of
this gradient flow structure are summarized in the following
result.

Before stating the Theorem we define the function $\omega:[0,+\infty)\to [0,+\infty)$
by \begin{equation}\label{omegadef}\omega(x)=\sqrt{\mass x\varphi(\mass^{-1}x)},\end{equation} where $\varphi$
is defined in \eqref{varphi}.
Moreover, given a fixed $s_0>0$, we define a strictly monotone continuous function
$G:[0,+\infty)\to [-\infty,+\infty)$ by
${G(s):=\int_{s_0}^s \frac{1}{\omega(r)}\,dr}$ for $s>0$ and $G(0)=-\infty$
(we observe that $G^{-1}: [-\infty,+\infty)\to [0,+\infty)$ is surjective).

\begin{theorem}\label{th:GFEVI}
Let $t\mapsto z_t=(\rho_t,v_t)$ be a bounded solution of problem
\eqref{ks} for $\eps>0$,  starting from $z^0=(\rho^0,v^0)\in
X_\eps \cap \big(L^\infty(\rd) \times (W^{1,2}(\rd)\cap
W^{2,\infty}(\rd))\big)$, according to {\rm Definition
\ref{def2}}. If $\eps=0$, let $z_t=\rho_t$ be a bounded solution
to problem \eqref{pe}, starting from $z^0=\rho^0\in X_0\cap
L^{\infty}(\rd)$, according to Definition {\rm \ref{def1}}. Then
the three following properties hold:
\begin{itemize}\item[\it i)]
The evolution variational inequality (EVI) formulation:
for any $\bar z=(\bar\rho,\bar v)\in X_\eps\cap \big(L^\infty(\rd) \times W^{1,2}(\rd)
\big)$ (reduced to $\bar z=\bar \rho\in X_0\cap L^\infty(\mathbb{R}^d)$ if $\eps=0$),
the map $t\mapsto D^2(z_t,\bar z)$ is absolutely continuous and there exists a constant $C$ depending on
$\|\rho\|_{L^{\infty}((0,T)\times\rd)}$,
$\|\bar\rho\|_{L^{\infty}(\rd)}$ and
$\|v^0\|_{W^{2,\infty}(\mathbb{R}^d)}$, such that
\begin{equation}\label{EVI}
    \frac{1}2\frac{d}{dt} D^2(z_t,\bar z) \le\,
\FF_{\eps,\alpha}(\bar z) - \FF_{\eps,\alpha}(z_t)+ C
\omega(D^2(z_t,\bar z))  \quad \text{for $\mathcal{L}^1$-a.e. }
t\in (0,T).
\end{equation}

 \item[\it ii)] The energy dissipation equality (EDE) in metric
sense: the map $t\mapsto \FF_{\eps,\alpha}(z_t)$ is locally
Lipschitz continuous and
\begin{equation}\label{EDE}
\frac{d}{dt}{\FF_{\eps,\alpha}}(z_t)=-\frac12
|\partial{\FF_{\eps,\alpha}}|_{D}^2(z_t)-\frac12|z^\prime|_{D}^2(t)
\quad \text{for $\mathcal{L}^1$-a.e. } t\in (0,T).
\end{equation}
\item[\it iii)] The following expansion control property:
given another bounded solution $t\mapsto\zeta_t$, with initial
datum $\zeta^0$ in the same space of $z^0$ above, there exists a constant $C$,  depending on  $\|\rho\|_{L^{\infty}((0,T)\times\rd)}$
 and
$\|v^0\|_{W^{2,\infty}(\mathbb{R}^d)}$ (and the same quantities associated to $\zeta$), such that there holds
\begin{equation}\label{contractivity}
D^2(z_t,\zeta_t)\le G^{-1}(G(D^2(z^0,\zeta^0))+4Ct) \quad \text{for every } t\in [0,T).
\end{equation}

\end{itemize}
\end{theorem}

\begin{proof}

We first introduce the auxiliary functional
\begin{equation*}
\Phi_{\eps,\alpha}(\rho,v):=
\int_{\mathbb{R}^d}(\rho\log\rho-v\rho)\,dx,
\end{equation*}
for $\rho$ and  $v$ being as in the definition of $\FF_{\eps,\alpha}$ at the beginning of this section, so that
\begin{equation*}
\FF_{\eps,\alpha}(\rho,v)=\Phi_{\eps,\alpha}(\rho,v)
+\frac12\int_{\mathbb{R}^d}(|\nabla v|^2+\alpha v^2)\, dx\,
\end{equation*}
and
\begin{equation}\label{renormphi}
\Phi_{0,0}(\rho,v) = \FF_{0,0}(\rho,v) -\frac12\int_{\mathbb{R}^d}|\nabla
(v-v^*)|^2\, dx+\frac12
\int_{\mathbb{R}^d}\rho^*v^*\,dx-\int_{\mathbb{R}^d}\!\rho^*v\,dx\,\quad\mbox{for $d=1,2$.}
\end{equation}

The proof is organized in four steps.

\vskip 6pt

{\it{Step1. Quasi-Lipschitz Estimate implies control of the
evolution of the Wasserstein distance.-}}
Thanks to the assumption \eqref{dissipation}, we learn that the Fisher
information
${\int_{\mathbb{R}^d}\frac{|\nabla\rho_t(x)|^2}{\rho_t(x)}\,dx}$ is finite for $\mathcal{L}^1$-a.e. $t\in (0,T)$. Let $\bar \rho\in \MeasuresTwo\cap L^{\infty}(\mathbb{R}^d)$. Exploiting the
differentiability properties of the entropy functional, we can use
the above-the-tangent formulation of displacement convexity to get
for $\mathcal{L}^1$-a.e. $t\in (0,T)$
\begin{equation}\label{wassersteindiff}
\int_{\mathbb{R}^d}\bar\rho(x)\log\bar\rho(x)\,dx-\int_{\mathbb{R}^d}\rho_t(x)\log\rho_t(x)\,dx\ge
\int_{\mathbb{R}^d}\langle\nabla\rho_t(x),\mathcal{T}_t(x)-x
\rangle\,dx,
\end{equation}
where $\mathcal{T}_t$ denotes the optimal transport map between
$\rho_t$ and $\bar\rho$. We refer to \cite[§3.3.1]{AG} for an intuitive
proof of this fact, and to \cite[Chapter 10]{AGS} for the theory
in full generality. In particular, the finiteness of the Fisher
information of $\rho_t$ implies that the second term is finite, so
that this differentiation formula is meaningful.
 If $\eps>0$ (resp. $\eps=0$), let $\bar v\in W^{1,2}(\mathbb{R}^d)$ (resp. $\bar v=\mathcal{B}_{\alpha,d}\ast\bar\rho$).
Take
$$
I_t:=\Phi_{\eps,\alpha}(\bar\rho,\bar
v)-\Phi_{\eps,\alpha}(\rho_t,v_t) +\int_{\mathbb{R}^d}(\bar v(x)-
v_t(x))\bar\rho(x)\,dx\,.
$$
Using the notation $x^s_t:=(1-s)x+s\mathcal{T}_t(x)$, $s\in [0,1]$,
and taking into account that
\begin{align*}
\int_{\mathbb{R}^d}v_t(x)(\bar\rho(x)-\rho_t(x))\,dx &= \int_{\mathbb{R}^d}(v_t(\mathcal{T}_t(x))-v_t(x))\rho_t(x)\,dx \\
&=\int_{\mathbb{R}^d}(v_t(x_t^1)-v_t(x_t^0))\rho_t(x)\,dx = \int_0^1 \frac{d}{ds} \int_{\mathbb{R}^d}v_t(x_t^s)\rho_t(x)\,dx\,ds
\end{align*}
and \eqref{wassersteindiff}, we obtain for $\mathcal{L}^1$-a.e.
$t\in (0,T)$
\begin{equation*}\label{Iestimate}
\begin{aligned}
I_t&\ge \int_{\mathbb{R}^d}\langle\nabla\rho_t(x),\mathcal{T}_t(x)-x
\rangle\,dx
-\int_{\mathbb{R}^d}v_t(x)(\bar\rho(x)-\rho_t(x))\,dx\\
&=  \int_{\mathbb{R}^d}\langle\nabla\rho_t(x),\mathcal{T}_t(x)-x
\rangle\,dx
-\int_0^1\int_{\mathbb{R}^d}\langle \nabla v_t(x^s_t),\mathcal{T}_t(x)-x\rangle\,\rho_t(x)\,dx\,ds\\
&=\int_{\mathbb{R}^d}\langle\nabla\rho_t(x)-\rho_t(x)\nabla v_t(x),\mathcal{T}_t(x)-x
\rangle\,dx -\int_0^1\int_{\mathbb{R}^d}\langle \nabla
v_t(x^s_t)-\nabla v_t(x),\mathcal{T}_t(x)-x\rangle\,\rho_t(x)\,dx\,ds.
\end{aligned}
\end{equation*}
Let us denote by $II_t$ the last term in the right hand side
above. The crucial point is to treat such term using the
log-Lipschitz property of $\nabla v$. Notice that, if $\eps=0$, we
are in the assumptions of Proposition \ref{pro1} and we apply
\eqref{loglipschitz}, where the constant $C$ depends in principle
only on ($\mass,\alpha, d$ and) the
$L^\infty$ norm of $\rho_t$, which we are assuming to be uniformly bounded on
$(0,T)$. In the case $\eps>0$, still by the uniform space-time
$L^\infty$ assumption on $\rho_t$ and the $W^{2,\infty}$
assumption on $v^0$, we are in the framework of Proposition
\ref{pro2}, so that we can apply the estimate
\eqref{parabolicloglipschitz}. In this case the constant will
depend also on ($\eps$ and) $\|v^0\|_{W^{2,\infty}(\mathbb{R}^d)}$.
 Since $\varphi$ is concave, we can
also use the Jensen inequality, and letting $\rho^s_t={x^s_t}
_\#\rho_t$ be the Wasserstein geodesic connecting $\rho_t$ and $\bar\rho$  we have
\begin{equation}\label{perdopo}
\begin{aligned}
\left|II_t\right|
 &\le  W_2(\rho_t,\bar\rho) \int_0^1\left(\int_{\mathbb{R}^d} |\nabla v_t(x^s_t)-\nabla v_t(x)|^2\rho_t(x)\,dx\right)^{1/2}\,ds\\
 &\le  C W_2(\rho_t,\bar\rho) \int_0^1\left(\int_{\mathbb{R}^d} \varphi(|x^s_t-x|^2)\rho_t(x)\,dx\right)^{1/2}\,ds\\
& \le \sqrt{\mass} C W_2(\rho_t,\bar\rho)
\int_0^1\sqrt{\varphi(\mass^{-1}W_2^2(\rho_t,\rho^s_t))}\,ds \leq \sqrt{\mass}C
W_2(\rho_t,\bar\rho) \sqrt{\varphi(\mass^{-1}W_2^2(\rho_t,\bar\rho))}\, .
\end{aligned}
\end{equation}
The last inequality holds since geodesic interpolation ensures
$$
\int_{\R^d} |x-x^s_t|^2\rho_t(x)\,dx =
W_2^2(\rho_t,\rho^s_t)=s^2W_2^2(\rho_t,\bar\rho)
$$
for all $s\in [0,1]$ and since $\varphi$ is non decreasing.
We recall that the constant $C$ in \eqref{perdopo} depends
only on ($\eps$, $\alpha, d$, the mass $\mass$ and) the
$L^\infty((0,T)\times\mathbb{R}^d)$ norm of $\rho$ and, in the case $\eps>0$, the $W^{2,\infty}(\mathbb{R}^d)$ norm of $v^0$.
Inserting
this in the estimate for $I_t$, we have for $\mathcal{L}^1$-a.e.
$t\in (0,T)$
\begin{equation}\label{I_t}
I_t\ge\int_{\mathbb{R}^d}\langle{\nabla\rho_t}(x)-\rho_t(x)\nabla
v_t(x),\mathcal{T}_t(x)-x \rangle\,dx - C \omega(W_2^2(\rho_t,\bar\rho))\,,
\end{equation}
where $\omega$ is the function defined in \eqref{omegadef}. Since
$\rho_t$ satisfies the continuity equation
$$
\partial_t\rho_t+\div(\xi_t\rho_t)=0 \qquad \mbox{with }\quad
\rho_t\xi_t = -{\nabla\rho_t}+\rho_t\nabla v_t
$$
and \eqref{dissipation}, the uniform $L^\infty$ bound of $\rho_t$
implies that $\int_0^T\|\xi_t\|^2_{L^2(\mathbb{R}^d,\rho_t;\mathbb{R}^d)}\,dt<+\infty$. Therefore
$t\mapsto\rho_t$ is absolutely continuous with respect to $W_2$
and by \eqref{timederivative}
\begin{equation*}
\frac12\frac{d}{dt}W_2^2(\rho_t,\bar\rho)=\int_{\mathbb{R}^d}\langle\nabla
\rho_t(x)-\rho_t(x)\nabla
v_t(x),\mathcal{T}_t(x)-x\rangle\,dx\qquad \mbox{for
$\mathcal{L}^1$-a.e. t}\in (0,T).
\end{equation*}
Inserting this into \eqref{I_t}, and recalling the definition of $I_t$, we finally obtain
\begin{equation}\label{evi2}
\frac12\frac{d}{dt}W_2^2(\rho_t,\bar\rho)\le
\Phi_{\eps,\alpha}(\bar\rho,\bar
v)-\Phi_{\eps,\alpha}(\rho_t,v_t)+\int_{\mathbb{R}^d}
(\bar v-v_t)\bar\rho\,dx + C
\omega(W_2^2(\rho_t,\bar\rho))
\end{equation}
for $\mathcal{L}^1$-a.e. $t\in (0,T)$.

\vskip 6pt

{\it{Step 2: EVI for the parabolic-parabolic case.-}}
Recalling that $\bar v\in W^{1,\,2}(\mathbb{R}^d)$,
observing that $\Delta v_t\in L^{2}(\mathbb{R}^d)$ for a.e.-$t\in (0,T)$ and using
the elementary identity $|a|^2-|b|^2=|a-b|^2 +2\la b, a-b \ra$ for every $a,b\in \R^k$,
 the variation of the second part of the functional
\eqref{fullfunctional} (that is, $\FF_{\eps,\alpha}-\Phi_{\eps,\alpha}$) can be written as
\begin{equation}\label{lunga}\begin{aligned}
\frac12\int_{\mathbb{R}^d}\left[|\nabla\bar v|^2-\right.&\left.|\nabla
v_t|^2+\alpha(\bar v^2- v_t^2)\right]\,dx
\\&=\int_{\mathbb{R}^d}(\alpha v_t-\Delta v_t)(\bar v-v_t)\,dx+\frac12 \|\nabla(v_t-\bar v)\|^2_{L^2(\mathbb{R}^d)}+\frac{\alpha}2\|v_t-\bar v\|^2_{L^2({\mathbb{R}^d})}\\
&=\int_{\mathbb{R}^d}(\rho_t-\eps\partial_t v_t)(\bar v-v_t)\,dx+\frac12 \|\nabla(v_t-\bar v)\|^2_{L^2(\mathbb{R}^d)}+\frac{\alpha}2\|v_t-\bar v\|^2_{L^2({\mathbb{R}^d})}\\
&=\int_{\mathbb{R}^d}\rho_t(\bar v-v_t)\,dx
+\frac{\eps}2\frac{d}{dt}\|v_t-\bar v\|_{L^2(\mathbb{R}^d)}^2
+\frac12 \|\nabla(v_t-\bar v)\|^2_{L^2(\mathbb{R}^d)}+\frac{\alpha}2\|v_t-\bar v\|^2_{L^2({\mathbb{R}^d})}.
\end{aligned}\end{equation}
Therefore, we deduce
\begin{equation}\label{kkk}
\begin{aligned}
\FF_{\eps,\alpha}(\bar\rho,\bar
v)-\FF_{\eps,\alpha}(\rho_t,v_t)&=\Phi_{\eps,\alpha}(\bar\rho,\bar
v)-\Phi_{\eps,\alpha}(\rho_t,v_t)+\int_{\mathbb{R}^d}\rho_t(\bar
v-v_t)\,dx\\&\quad+\frac{\eps}2\frac{d}{dt}\|v_t-\bar
v\|_{L^2(\mathbb{R}^d)}^2+\frac12 \|\nabla(v_t-\bar v)\|^2_{L^2(\mathbb{R}^d)}+\frac{\alpha}2\|v_t-\bar v\|^2_{L^2({\mathbb{R}^d})}.
\end{aligned}
\end{equation}
Now, we use again \eqref{evi2},
leading to
\begin{equation}\label{ev}
\begin{aligned}
\frac{\eps}2\frac{d}{dt}\|v_t-\bar
v\|_{L^2(\mathbb{R}^d)}^2&+\frac12\frac{d}{dt}W_2^2(\rho_t,\bar\rho)\le\,
\FF_{\eps,\alpha}(\bar\rho,\bar
v)-\FF_{\eps,\alpha}(\rho_t,v_t)+ C
\omega(W_2^2(\rho_t,\bar\rho))\\
&\quad+ \int_{\mathbb{R}^d} (\bar \rho -\rho_t)(\bar v-v_t)\,dx-\frac12 \|\nabla(v_t-\bar v)\|^2_{L^2(\mathbb{R}^d)}-\frac{\alpha}2\|v_t-\bar v\|^2_{L^2({\mathbb{R}^d})}.
\end{aligned}
\end{equation}
By using the duality
between $\dot H^1$ and $\dot H^{-1}$, the Young inequality, and
\eqref{h-1} we have
\begin{equation}\label{e}
\begin{aligned}
\int_{\mathbb{R}^d} (\bar \rho -\rho_t)(\bar v-v_t)\,dx&\le \|\bar \rho -\rho_t\|_{\dot H^{-1}(\mathbb{R}^d)}\|\bar v-v_t\|_{\dot H^1(\mathbb{R}^d)}
\le \frac12\|\bar \rho-\rho_t\|^2_{\dot H^{-1}(\mathbb{R}^d)}+\frac12\|\nabla(\bar v-v_t)\|^2_{L^2(\mathbb{R}^d)}\\
&\le \frac12 QW_2^2(\bar\rho,\rho_t)+\frac12\|\nabla(\bar v-v_t)\|^2_{L^2(\mathbb{R}^d)},
\end{aligned}
\end{equation}
where $Q$ is  the largest of the $L^\infty$
norms of $\bar\rho$ and $\rho_t$ over the time interval $(0,T)$.
Taking into account that $\omega$ is given by \eqref{omegadef} and  that $\sqrt{\mass\varphi(\mass^{-1}x^2)}\ge x$ for every $x>0$,
 combining \eqref{ev} and \eqref{e} we get, up to introducing a new
constant $C$,
\begin{equation}\label{generalEVI}
\begin{aligned}
\frac{\eps}2\frac{d}{dt}\|v_t-\bar
v\|_{L^2(\mathbb{R}^d)}^2+\frac12\frac{d}{dt}W_2^2(\rho_t,\bar\rho)&\le\,
\FF_{\eps,\alpha}(\bar\rho,\bar
v)-\FF_{\eps,\alpha}(\rho_t,v_t)+ C
\omega(W_2^2(\rho_t,\bar\rho))
-\frac\alpha2 \|v_t-\bar v\|^2_{L^2({\mathbb{R}^d})}
\end{aligned}
\end{equation}
for a.e. $t\in(0,T)$. The new constant  $C$ depends as usual  on ($\eps$
$\alpha$, $d$, $\mass$ and) $\|\rho\|_{L^\infty((0,T)\times{\mathbb{R}^d})}$,
$\|v^0\|_{W^{2,\infty}(\mathbb{R}^d)}$, $\|\bar\rho\|_{L^\infty(\mathbb{R}^d)}$.

\vskip 6pt

{\it{Step 3: EVI for the parabolic-elliptic case.-}} When either
$d\ge 3$ or $\alpha>0$, we can repeat the proof of the
parabolic-parabolic case, letting $\eps=0$ therein and recalling
that $\bar v$ is no more an arbitrary $W^{1,2}(\mathbb{R}^d)$
function but is given by convolution with $\bar\rho$. In
particular we arrive to the corresponding of \eqref{ev}, and the
second line therein can now be estimated as follows. Using the
inequality $\|v\|_{H^1_\alpha(\mathbb{R}^d)}\leq
\|\rho\|_{H^{-1}_\alpha(\mathbb{R}^d)}$ for $-\Delta v+\alpha
v=\rho$, $\alpha>0$, where  the notation is
$\|v\|^2_{H^1_\alpha(\mathbb{R}^d)}:=\|\nabla
v\|^2_{L^2(\mathbb{R}^d)}+\alpha\|v\|^2_{L^2(\mathbb{R}^d)}$ (and
using $\dot H^1$ if $\alpha=0$),
 we get
\begin{equation*}
\int_{\mathbb{R}^d}(\bar v- v_t)(\bar\rho-\rho_t)\,dx\!\le
\|\bar v-v_t\|_{H^1_\alpha(\mathbb{R}^d)}\|\bar \rho-\rho_t\|_{H^{-1}_\alpha(\mathbb{R}^d)}\le
\frac12\|\bar v-v_t\|_{H^1_\alpha(\mathbb{R}^d)}^2+\frac12\|\bar \rho-\rho_t\|_{H^{-1}_\alpha(\mathbb{R}^d)}^2 .
\end{equation*}
Moreover, recalling the estimate
\eqref{h-1}  (which works both in $\dot H^{-1}$ and $ H^{-1}_\alpha$)
we have
\[\|\bar \rho-\rho_t\|_{H^{-1}_\alpha(\mathbb{R}^d)}\le
Q W_2^2(\bar\rho,\rho_t),\]
for all $t \in [0,T]$, where $Q$ is the largest of the $L^\infty$
norms of $\bar\rho$ and $\rho_t$ over the time interval $[0,T]$. Inserting these estimates in \eqref{ev} we obtain
\begin{equation}\label{generalEVI2}
\begin{aligned}
\frac12\frac{d}{dt}W_2^2(\rho_t,\bar\rho)\le\,
\FF_{0,\alpha}(\bar\rho)-\FF_{0,\alpha}(\rho_t)+ C
\omega(W_2^2(\rho_t,\bar\rho)),
\end{aligned}
\end{equation}
for $\mathcal{L}^1$-a.e. $t\in(0,T)$, where the constant $C$ depends only on $\eps$,
$\alpha$, $d$, $\mass$, $\|\rho\|_{L^\infty((0,T)\times{\mathbb{R}^d})}$,
$\|\bar\rho\|_{L^\infty(\mathbb{R}^d)}$.

In the case $\alpha=0$, $d=1,2$, we have to consider the functional in \eqref{renorm}. By using the identity
$$
\frac12 \|\nabla(\bar v-v^*)\|^2_{L^2(\mathbb{R}^d)} - \frac12
\|\nabla(v_t-v^*)\|^2_{L^2(\mathbb{R}^d)} =
\int_{\mathbb{R}^d}(\rho_t-\rho^*)(\bar v-v_t)\,dx +\frac12
\|\nabla(v_t-\bar v)\|^2_{L^2(\mathbb{R}^d)} ,
$$
with similar computations as in \eqref{lunga}, this time considering $\FF_{0,0}(\rho,v)-\Phi_{0,0}(\rho,v)$ as obtained from \eqref{renormphi}, we can still find
\eqref{kkk} and conclude obtaining again \eqref{generalEVI2}.

\vskip 6pt

{\it{Step 4: Conclusion.-}}
We are ready to prove the three points in the statement of the theorem.
The proof of {\it i)} is a consequence of
\eqref{generalEVI}
for the case $\eps>0$, and \eqref{generalEVI2} for the case $\eps=0$,
taking into account that
$\alpha \geq 0$ and that $\omega(D^2(z_t,\bar z)) \geq
\omega(W_2^2(\rho_t,\bar \rho))$ being $\omega$ increasing.

It is a standard fact that the gradient flow formulation in EVI sense
implies the one in EDE sense in \eqref{EDE}. Indeed, the proof of
{\it{ ii)}} follows from \eqref{EVI} and \eqref{contractivity} and
can be exactly carried out as in \cite[Proposition 3.6]{AG}.

The proof of \eqref{contractivity} still follows from \eqref{EVI}.
Indeed we can apply \cite[Lemma 4.3.4]{AGS} (see also the argument of \cite[Theorem 11.1.4]{AGS})
and obtain that for $\mathcal{L}^1$-a.e.  $t\in (0,T)$
\begin{equation}\label{contr1}
    \frac{1}2\frac{d}{ds} D^2(z_s,\zeta_s){\Big{|}}_{s=t} \le\,
    \frac{1}2\frac{d}{ds} D^2(z_s,\zeta_t){\Big{|}}_{s=t} + \frac{1}2\frac{d}{ds} D^2(z_t,\zeta_s){\Big{|}}_{s=t}
  \le 2C\omega(D^2(z_t,\zeta_t)).
\end{equation}
Here, $C=\max\{C_1,C_2\}$, where $C_1$ is the supremum on $s\in(0,T)$ of the constant in \eqref{EVI} for $z_t$ with $\bar z=\zeta_s$, which is finite since $\zeta\in L^\infty((0,T)\times\mathbb{R}^d)$, and $C_2$ is the same inverting $z$ and $\zeta$.
The estimate \eqref{contr1}   implies
\begin{equation*}
    \frac{d}{dt} D^2(z_t,\zeta_t) \le\,
 4 C\omega(D^2(z_t,\zeta_t)), \quad \text{for $\mathcal{L}^1$-a.e. } t\in (0,T).
\end{equation*}
Since the inequality
\begin{equation*}
     y(t) \le y(0)+ 4 C \int_0^t\omega(y(s))\,ds
\end{equation*}
entails that $y(t) \leq G^{-1}(G(y(0))+4Ct)$, we conclude.
\end{proof}

\begin{proofad}
The main theorems in the introduction are now a straightforward
consequence of the expansion control {\it iii)} in Theorem
\ref{th:GFEVI}. Both Theorems follow from the inequality
\eqref{contractivity} observing that
$G^{-1}(G(0)+4Ct)=G^{-1}(-\infty)=0$.
\end{proofad}


\section{$\omega$-convexity of the functional}\label{sec:convexity}

In this section we show another consequence of the EVI
formulation of bounded solutions. For the functional
$\FF_{\eps,\alpha}$  the following relaxed $\omega$-convexity
along geodesics holds, see \cite{CMV} for $\omega$-convexity of
functionals on measures. We assume that bounded solutions to
\eqref{ks} (resp. \eqref{pe} for $\eps=0$) verify that for some
$T>0$
\begin{equation}\label{linftyesti}
\|\rho_t\|_{L^\infty(\R^d)} \leq R_T(\rho^0,v^0):=R
\left(T,\|\rho^0\|_{L^\infty(\R^d)},\FF_{\eps,\alpha}(\rho^0,v^0)\right),
\quad \text{for $\mathcal{L}^1$-a.e. } t\in (0,T)\,.
\end{equation}
This assumption has been proved in several cases, see the
introduction for more details.

\begin{theorem}\label{th:convexity}
Assume that bounded solutions for the evolutions \eqref{ks} (resp.
\eqref{pe} for $\eps=0$) exist and verify \eqref{linftyesti}.
Then, for every $z^0,z^1 \in X_\eps \cap \big(L^\infty(\rd) \times
(W^{1,2}(\rd)\cap W^{2,\infty}(\rd))\big)$ (reduced to $X_0\cap
L^\infty(\mathbb{R}^d)$ if $\eps=0$) and every geodesic
$s\in[0,1]\to z^s$ of the space $(X_\eps,D)$, connecting $z^0$ to
$z^1$ there holds for all $s\in [0,1]$
\[
\FF_{\eps,\alpha}(z^s)\le
(1-s)\FF_{\eps,\alpha}(z^0)+s\FF_{\eps,\alpha}(z^1)+R_T\left[(1-s)\,\omega(s^2D^2(z^0,z^1))+s\,\omega((1-s)^2D^2(z^0,z^1))\right],
\]
where $R_T:=\max(R_T(z^0),R_T(z^1))$.
\end{theorem}

\begin{proof}
We first remark that the set $X_\eps \cap \big(L^\infty(\rd)
\times (W^{1,2}(\rd)\cap W^{2,\infty}(\rd))\big)$ (resp. $X_0\cap
L^\infty(\mathbb{R}^d)$ if $\eps=0$) is geodesically convex. This
is trivial for the part of the functional concerning $v$ while for
the density $\rho$ we use the classical displacement convexity of
all the $L^p$ norms \cite{Mc}. Now, we take $X_\eps \cap
\big(L^\infty(\rd) \times (W^{1,2}(\rd)\cap
W^{2,\infty}(\rd))\big)$ (resp. $X_0\cap L^\infty(\mathbb{R}^d)$
if $\eps=0$) as the set of initial data for the evolutions
\eqref{ks} (resp. \eqref{pe}).

Consider any $\bar z\in X_\eps\cap \left(L^\infty(\mathbb{R}^d)\times W^{1,2}(\mathbb{R}^d)\right)$ (reduced to $\bar z\in X_0\times L^\infty(\mathbb{R}^d)$ if $\eps=0$).
Taking into account that for a bounded solution $t\mapsto
D^2(z_t,\bar z)$ is absolutely continuous and $t\mapsto
\FF_{\eps,\alpha}(z_t)$ is decreasing by \eqref{EDE}, from
\eqref{EVI} we obtain
\begin{equation}\label{EVI2}
    \frac{1}2 D^2(z_t,\bar z) - \frac{1}2 D^2(z_0,\bar z) \le\,
t(\FF_{\eps,\alpha}(\bar z) - \FF_{\eps,\alpha}(z_t)) + C
\int_0^t\omega(D^2(z_r,\bar z))\,dr
\end{equation}
for all $t\in [0,T]$. We denote by $z_t^s$ the
bounded solution of \eqref{ks} or \eqref{pe} starting from the
initial datum $z^s$. We multiply by $(1-s)$ the inequality in
\eqref{EVI2} for $z_t=z_t^s$ and $\bar z= z^0$ and we multiply by $s$
the inequality in \eqref{EVI2} for $z_t=z_t^s$ and $\bar z= z^1$.
Summing up the two inequalities we obtain
\begin{align*}
\frac{1}2 ((1-s)D^2(z^s_t,z^0)+sD^2(z^s_t,z^1)) &- \frac{1}2 ((1-s)D^2(z^s,z^0) +sD^2(z^s,z^1))  \\
 \le\,&
t((1-s)\FF_{\eps,\alpha}(z^0) +s\FF_{\eps,\alpha}(z^1) - \FF_{\eps,\alpha}(z^s_t)) \\
&+ C((1-s)\int_0^t\omega(D^2(z^s_r,z^0))\,dr +s\int_0^t\omega(D^2(z^s_r,z^1))\,dr).
\end{align*}
Using the fact that $s\mapsto z^s$ is a geodesic, the right hand
side is nonnegative, thus
\begin{align*}
\FF_{\eps,\alpha}(z^s_t) -(1-s)\FF_{\eps,\alpha}(z^0)& - s\FF_{\eps,\alpha}(z^1) \\
&\le C((1-s)\frac1t\int_0^t\omega(D^2(z^s_r,z^0))\,dr +s\frac1t\int_0^t\omega(D^2(z^s_r,z^1))\,dr).
\end{align*}
The lower semi continuity of $t\mapsto \FF_{\eps,\alpha}(z^s_t)$
and the continuity of $r\mapsto D^2(z^s_r,z^i)$, $i=0,1$ yield
$$
\FF_{\eps,\alpha}(z^s)\le (1-s)\FF_{\eps,\alpha}(z^0)+s\FF_{\eps,\alpha}(z^1)+C((1-s)\omega(D^2(z^s,z^0))+s\omega(D^2(z^s,z^1))).
$$
Since $s\mapsto z^s$ is a geodesic we have $D^2(z^s,z^0)=
s^2D^2(z^1,z^0)$ and $D^2(z^s,z^1)= (1-s)^2D^2(z^0,z^1)$ and we
conclude.
\end{proof}


\section{A refined result in Zygmund spaces}\label{sec:zygmund}
This section is devoted to give a rigorous justification of the
estimates stated in subsection 2.1. We will also give a slight
improvement of  Theorem \ref{th:GFEVI}
and Theorem \ref{main2} by guaranteing a suitable
quasi-Lipschitz estimate under a more general condition on
the initial datum $v^0$. The right framework is that of Zygmund spaces.
These classes of functions were introduced in \cite{Z1}, and they
belong to the more general framework of Besov spaces.

\subsubsection*{Zygmund estimates and log-Lipschitz regularity in the elliptic case}
We first  introduce the basic Zygmund class
$\Lambda_1({\mathbb{R}^d})$ as the set of continuous bounded
functions $f$ over $\mathbb{R}^d$ such that
\[
\sup_{x,y\in\mathbb{R}^d} \frac{|f(x)-2f((x+y)/2)+f(y)|}{|x-y|}<+\infty.
\]
It is well known that functions in the Zygmund class are in
general not Lipschitz, possibly nowhere differentiable.  Indeed,
functions in $\Lambda_1(\mathbb{R}^d)$  enjoy a log-Lipschitz
modulus of continuity. Therefore, for any $f\in
\Lambda_1(\mathbb{R}^d)$ there exists a positive constant $C$ such
that
\[
|f(x)-f(y)|\le C |x-y||\log|x-y||\quad\forall x,y\in\mathbb{R}^d,
\]
we refer for instance to \cite[Chapter 2, §3]{Z2}. More generally,
following for instance \cite[Chapter 5]{S} we may define the class
$\Lambda_n(\mathbb{R}^d)$ for any $n\in\mathbb{N}$ as follows. We
let $\Lambda_0=L^∞(\mathbb{R}^d)$ and we say that $f\in
\Lambda_0(\mathbb{R}^d)$ belongs to $\Lambda_n(\mathbb{R}^d)$,
$n\ge 2$, if $\nabla f\in \Lambda_{n-1}(\mathbb{R}^d)$ or,
equivalently, if $f\in W^{n-1,∞}(\mathbb{R}^d)$ and all the
derivatives of $f$ of order $n-1$ belong to
$\Lambda_1(\mathbb{R}^d)$. In the usual notation of Besov spaces,
$\Lambda_n$ corresponds to $B^n_{∞,\,∞}$.
In this framework we have
\begin{proofad1}
If $\alpha>0$, from the general theory on Bessel potentials (see
for instance \cite[Chapter 5, §3-6]{S}) we learn that by
convolution with the Bessel kernel $\mathcal{B}_{\alpha,d}$ we
indeed get two indices of regularity in $\Lambda_n$ spaces.
Therefore, if $\rho\in L^\infty(\mathbb{R}^d)$, we indeed get that
$v=\mathcal{B}_{\alpha,d}\ast \rho$ belongs to
$\Lambda_2(\mathbb{R}^d)$, and thus $\nabla v\in
\Lambda_1(\mathbb{R}^d)$ and \eqref{loglipschitz} follows. For the
case $\alpha=0$ we address to the references mentioned in Section
\ref{prel} (it is also possible to directly check that $\nabla
v\in L^\infty(\mathbb{R}^d)$, and then the Newtonian potential
behaves like the Bessel potential near the origin so that $\nabla
v$ is also log-Lipschitz).
\end{proofad1}

\subsubsection*{Zygmund estimates and log-Lipschitz regularity in the parabolic case}

Let $T>0$. In this section,  let us denote  $Q_T:=(0,T)\times \mathbb{R}^d$ and then
$\bar Q_T:=[0,T]\times\mathbb{R}^d$. In the half $d+1$ dimensional
space, we consider the standard parabolic metric
\[
\delta((x,t),(y,s)):=\max\{|x-y|,\sqrt{|t-s|}\}.
\]
With respect to the parabolic metric, the definition of Zygmund
spaces adapts as follows. We have $\Lambda_0(\bar Q_T):=L^∞(Q_T)$,
and  $\Lambda_1(\bar Q_T)$ is the space of continuous bounded
functions $f$ over $\bar Q_T$ such that there hold
\begin{equation}\label{ell1}
\sup_{{x,y\in\mathbb{R}^d}\atop{ t \in [0,T] }} \!\!\frac{|f(x,t)-2f((x+y)/2,t)+f(y,t)|}{|x-y|}\; +\!\!\!\sup_{{x\in\mathbb{R}^d}\atop{ 0\le s< t\le T} }\!\!\!
\frac{|f(x,t)-2f(x,(t+s)/2)+f(x,s)|}{|t-s|^{1/2}}<+∞.
\end{equation}
Moreover, we say that $f\in L^\infty(Q_T)$ belongs to
$\Lambda_2(\bar Q_T)$ if
\begin{equation*}
\sup_{x\in\mathbb{R}^d\,,\, 0\le s< t\le T } \frac{|f(x,t)-2f(x,(t+s)/2)+f(x,s)|}{|t-s|}<+∞
\end{equation*}
and $\nabla f\in\Lambda_1(\mathbb{R}^d)$. In particular,  we see
that $f\in\Lambda_2(\bar Q_t)$ implies $f\in
L^∞((0,T);W^{1,∞}(\mathbb{R}^d))$, with $\nabla f$ satisfying
\eqref{ell1}, so that finally $f$ satisfies also
\eqref{parabolicloglipschitz}.

When dealing with parabolic equations, it is suitable to consider
spaces of functions defined with respect to the parabolic metric,
since it is natural to deal with functions which have derivative
up to order $k$ with respect to time and $2k$ with respect to
space. For classic results, we refer for instance to \cite{C} or
to the monograph \cite{LSU}, where estimates are derived in
Sobolev and H\"older spaces of this kind, see Chapter 4 therein.

In \cite{C} we find that if the forcing term of the heat equation
has bounded mean oscillation (BMO), still with respect to the
parabolic metric, than the same holds true for second order space
derivatives and first order time derivatives of the solution. This
would be enough for deducing that first derivatives in space are
in the Zygmund class with respect to the parabolic metric and that
therefore they satisfy a log-Lipschitz estimate. The results in
\cite{C} deal only with null initial datum, but they can be
generalized to more general data with suitable regularity
requirements. Some extensions involving initial data in Zygmund
classes are found in \cite{A,K}, based on direct estimates on
fundamental solutions. Summing up, we have
\begin{proofad2}
Suppose that $v$ is the solution (convolution with fundamental
operator) of the forced heat equation $\partial_t v=\Delta
v+\rho$. Suppose $\rho\in \Lambda_0(\bar Q_T)$ and $v^0\in
\Lambda_{2}(\mathbb{R}^d)$. Then we have $v\in\Lambda_{2}(\bar
Q_T)$. See \cite{C} for the case $v^0=0$, see \cite[Theorem 4]{K}
for a general result. If we consider the second equation of
\eqref{ks} with $\alpha>0$,  the fundamental solution is just
multiplied by a decaying exponential at infinity and the same
result carries over. Therefore, a sufficient condition in order to
have $v$ satisfying \eqref{parabolicloglipschitz} is $v^0\in
W^{2,\infty}(\mathbb{R}^d)$, because
$W^{2,\infty}(\mathbb{R}^d)\subset \Lambda_2(\mathbb{R}^d)$.
\end{proofad2}

This gives a rigorous justification of the assumptions on the
initial datum of Theorem \ref{main2}. However a refined analysis
shows that this assumption can be weakened, as we do next.


\subsubsection*{Initial datum in $\Lambda_1(\mathbb{R}^d)$}
We have to consider the weighted
Zygmund space $\Lambda_2^{-1}(Q_T)$, defined as the
corresponding space $\Lambda_2(\bar Q_T)$, with the addition of a
time weight which is divergent as $t\to 0$. In particular, locally
in $Q_T$ functions in  $\Lambda_2^{-1}(Q_T)$ have the same
smoothness as the ones in $\Lambda_2(\bar Q_T)$, but this
regularity does no more extend to the closure of $Q_T$. More
precisely, by definition $f\in\Lambda_2^{-1}(Q_T)$ means that
$f\in\Lambda_1(\bar Q_T)$,
\begin{equation}\label{atzero}
\sup_{{x,y\in\mathbb{R}^d}\atop{ t \in [0,T] }} \sqrt{t}\; \frac{|\nabla f(x,t)-2\nabla f((x+y)/2,t)+\nabla f(y,t)|}{|x-y|}<+\infty
\end{equation}
and the second finite differences of $f$ and $\nabla f$ with
respect to time verify the corresponding estimates, as in the
definition of $\Lambda_2(\bar Q _T)$, still with the addition of
the weight $t^{1/2}$.

\begin{theorem}\label{refined}
Let $T>0$. Let $\rho^0\in\MeasuresTwo \cap  L^\infty(\mathbb{R}^d)$ and
$v^0\in \Lambda_1(\mathbb{R}^d)\cap W^{1,2}(\mathbb{R}^d)$.
Let $z_t=(\rho_t,v_t)$ be
a bounded solution on $[0,T]\times\mathbb{R}^d$ to the Cauchy
problem for \eqref{ks}, according to {\rm Definition \ref{def2}}, with
initial datum $z^0=(\rho^0,v^0)$.
%
For any reference point $\bar z=(\bar\rho,\bar v)\in(\MeasuresTwo \cap L^\infty(\mathbb{R}^d))\times  W^{1,2}(\mathbb{R}^d)$,
the general EVI holds
\begin{equation}\label{EVIZ}
    \frac{1}2\frac{d}{dt} D^2(z_t,\bar z) \le\,
\FF_{\eps,\alpha}(\bar z) - \FF_{\eps,\alpha}(z_t)+ C  t^{-1/2}
\omega(D^2(z_t,\bar z))  \quad \text{for $\mathcal{L}^1$-a.e. }
t\in (0,T),
\end{equation}
for a constant $C$ depending  on $\|\rho\|_{L^\infty((0,T)\times{\mathbb{R}^d})}$,
$\|v^0\|_{\Lambda^{1}(\mathbb{R}^d)}$, $\|\bar\rho\|_{L^\infty(\mathbb{R}^d)}$.

Moreover the EDE \eqref{EDE} holds, and the expansion control property holds in this form:
given another bounded solution $t\mapsto\zeta_t$ as above with initial
datum $\zeta^0 \in(\MeasuresTwo \cap L^\infty(\mathbb{R}^d))\times (\Lambda_1(\mathbb{R}^d)\cap W^{1,2}(\mathbb{R}^d))$ there is
\begin{equation}\label{contractivityZ}
D^2(z_t,\zeta_t)\le G^{-1}(G(D^2(z^0,\zeta^0))+8C\sqrt t) \quad \text{for every } t\in [0,T),
\end{equation}
where $C$ is a constant depending on   $\|\rho\|_{L^{\infty}((0,T)\times\rd)}$
 and
$\|v^0\|_{\Lambda_1(\mathbb{R}^d)}$ (and the same quantities associated to $\zeta$). In particular, $z=\zeta$ if $z^0=\zeta^0$.
\end{theorem}

\begin{proof}
Since we are in the hypotheses of \cite[Theorem 4]{K}, $v$
belongs to $\Lambda_2^{-1}(Q_T)$, so that \eqref{atzero} above
holds for $v$ and then,  due to the log-Lipschitz regularity in
the Zygmund class, we deduce
\begin{equation}\label{loggy}
|\nabla v_t(x)-\nabla v_t(y)|\le K t^{-1/2}|x-y||\log|x-y||,
\end{equation}
for all $x\in\mathbb{R}^d$, $t\in (0,T)$, where $K$ is a suitable
constant depending only on $T$ and the data. Notice that from the
definition of $\Lambda_2^{-1}(Q_T)$, it does not follow
that $\nabla v\in L^∞(Q_T)$. However, in \cite[Theorem 4]{K} it
is also shown that, still for $v^0\in \Lambda_1(\mathbb{R}^d)$ and
$\rho\in L^∞(Q_T)$, the solution $v_t$ of the
parabolic equation also satisfies
\begin{equation}\label{atinfty}
\|\nabla v_t(\cdot)\|_{L^∞(\mathbb{R}^d)}\le \bar K(1+|\log t|),\quad t\in(0,T),
\end{equation}
again for some positive $\bar K$ depending on  $\|\rho\|_{L^\infty(Q_T)}$ and $\|v^0\|_{\Lambda_1(\mathbb{R}^d)}$.
Taking  \eqref{atinfty} into account, it is clear that \eqref{loggy}
can be improved (we do not relabel the constant) into
\[
|\nabla v_t(x)-\nabla v_t(y)|\le K t^{-1/2}|x-y|(1+\log^-|x-y|).
\]
Thus we deduce the weighted analogous of \eqref{parabolicloglipschitz}, that is
\begin{equation}\label{finalloglipschitz}
|\nabla v_t(x)-\nabla v_t(x)|^2\le \frac{C^2}{t}\:\varphi(|x-y|^2),
\end{equation}
where $C$ is a new suitable positive constant depending on the
data and $\varphi$ is defined in \eqref{varphi}. Following the
line of the proof Theorem \ref{th:GFEVI} we reach
the estimate \eqref{perdopo} for $II_t$, which now has to be
changed because we have to use \eqref{finalloglipschitz},
obtaining
\[
|II_t|\le Ct^{-1/2} W_2(\rho_t,\bar\rho)
\sqrt{\mass\varphi(\mass^{-1} W_2^2(\rho_t,\bar\rho))}=Ct^{-1/2}\omega(W_2^2(\rho_t,\bar\rho)).
\]
We can repeat all the other steps which lead to \eqref{e},
obtaining the corresponding EVI with the additional weight
$t^{-1/2}$, which directly lead to \eqref{EVIZ}. We conclude as in
Step 4 of the proof of Theorem \ref{th:GFEVI}: from \eqref{EVIZ},
the EDE formulation \eqref{EDE} follows, still referring to
\cite[Proposition 3.6]{AG}. Moreover, \eqref{contractivityZ}
follows by \eqref{EVIZ} by
\begin{equation*}
    \frac{d}{dt} D^2(z_t,\zeta_t) \le\,
 4 C t^{-1/2} \omega(D^2(z_t,\zeta_t)), \quad \text{for $\mathcal{L}^1$-a.e. } t\in (0,T).
\end{equation*}
Indeed the inequality
\begin{equation*}
     y(t) \le y(0)+ 4 C \int_0^t s^{-1/2}\omega(y(s))\,ds
\end{equation*}
implies that $y(t) \leq G^{-1}(G(y(0))+8C\sqrt t)$ as desired. Finally, the uniqueness result follows since $G(0)=-\infty$ and $G^{-1}(-\infty)=0$.
\end{proof}


\section{The case of nonlinear diffusion}\label{sec:nonlinear}

We show next how to adapt our techniques to more general
aggregation diffusion equations in a quite straightforward way.
Let us consider the problem
\begin{equation}\label{nonlinearpp}
\left\{\begin{array}{rl}
\partial_t \rho&=\div (\rho\nabla P(\rho))-\, \div(\rho\nabla
v),\\[2mm]
\eps\partial_t v&=\Delta v + \rho- \alpha v,
\end{array}\right.
\end{equation}
to which we associate the functional
\begin{equation}\label{generalizedfunctional}
\GG_{\eps,\alpha}(\rho,v):=\int_{\mathbb{R}^d}(\Psi(\rho)-v\rho)\,dx+\frac12\int_{\mathbb{R}^d}(|\nabla
v|^2+\alpha v^2)\,dx,
\end{equation}
for all $\eps>0$, $\alpha\geq 0$, $\rho\in\MeasuresTwo\cap L^\infty({\mathbb{R}^d})$, $v\in W^{1,2}(\mathbb{R}^d)$, where ${\Psi(\rho):=\int_0^\rho
P(r)\,dr}$. \BBB We give the same restrictions as \cite[§9.3]{AGS}, the
first one being
$$
\lim_{r\to 0}\frac{\Psi(r)}{r^q}>-∞\quad\mbox{for some }
q>\frac{d}{d+2},
$$
a property ensuring that ${\int_{\mathbb{R}^d}\Psi(\rho)\neq -∞}$.
Moreover, the crucial property to be satisfied by the new
nonlinearity is the displacement convexity, that is the map
$r\mapsto r^{d}\Psi(r^{-d})$ is convex and nondecreasing on $(0,+∞)$. This
notion, introduced in \cite{Mc}, is stronger than convexity and
corresponds for $C^2$ functions to the inequality
\[
r^{-1}\Psi(r)-\Psi^\prime(r)+r\Psi^{\prime\prime}(r)\ge-\,\frac{1}{d-1}\,r
\Psi^{\prime\prime}(r) \qquad \forall\;r\in(0,+\infty).
\]
The more relevant cases correspond to nonlinear diffusion of
power kind. Indeed, if
$$
\Psi(\rho)=\frac{1}{m-1}\,\rho^m,\quad m\ge \frac{d-1}{d}
$$
the displacement convexity property holds. The case $m>1$ (resp.
$m<1$) correspond to a slow diffusion (resp. fast diffusion) in
the equation. On the other hand, the linear diffusion is recovered
taking $P(\rho)=\log\rho$, it is seen that in this case functional
\eqref{generalizedfunctional} is reduced, up to a constant, to
\eqref{fullfunctional}. Finally, let us mention that the
free-energy functional in the parabolic-elliptic case is similar
to \eqref{functional2} and given by
\begin{equation}\label{generalizedfunctional2}
\GG_{0,\alpha}(\rho,v):=\int_{\mathbb{R}^d}(\Psi(\rho)-\frac12
v\rho)\, dx,
\end{equation}
for $\rho\in\MeasuresTwo\cap L^\infty(\mathbb{R}^d)$ and $v=\mathcal{B}_{\alpha,d}\ast \rho$. It can be written as \eqref{generalizedfunctional}, taking into account the same renormalization as in \eqref{renorm}, to be done in the
pathological cases $\eps=\alpha=0$ and $d=1,2$.

The notion of bounded solution is completely analogous to
Definitions \ref{def1} and \ref{def2}, both for the
parabolic-elliptic and the parabolic-parabolic case. Indeed, the
only point to adapt is the finiteness  of the Fisher
information, now rewritten into the generalized version
\begin{equation}\label{generalizedfisher}
\int_0^T\int_{\mathbb{R}^d}|\nabla
P(\rho_t(x))|^2\,\rho_t(x)\,dx\,dt<+∞.
\end{equation}

\begin{corollary}
{\rm Theorem \ref{main1}}, {\rm Theorem \ref{main2}}, {\rm Theorem \ref{th:GFEVI}} and {\rm
Theorem \ref{refined}} hold for bounded solutions to
\eqref{nonlinearpp}.
\end{corollary}

\begin{proof}
The displacement convexity property makes the internal energy
functional $\rho\in\MeasuresTwo \mapsto
\int_{\mathbb{R}^d}\Psi(\rho(x))\,dx$ convex along Wasserstein
geodesics, as shown in \cite[§9.3]{AGS}. This in turn gives the
possibility to write down a subdifferential inequality in
Wasserstein sense (for a definition see \cite[§10.1.1]{AGS}) as
follows. Let $\rho\in\MeasuresTwo\cap L^{\infty}(\mathbb{R}^d)$ be
such that $\int_{\mathbb{R}^d}|\nabla P(\rho)|^2\,\rho\,dx$ is
finite. Then
\begin{equation}\label{internalenergysubdiff}
\int_{\mathbb{R}^d}
\Psi(\bar\rho(x))\,dx-\int_{\mathbb{R}^d}\Psi(\rho(x))\,dx\ge
\int_{\mathbb{R^d}}\langle \nabla
P(\rho(x)),\mathcal{T}(x)-x\rangle\,\rho(x)\,dx,
\end{equation}
for any $\bar\rho\in\MeasuresTwo$, where $\mathcal T$ is the
optimal transport map from $\rho$ to $\bar\rho$. Convexity and
differentiability of functionals defined on probability densities,
as the internal energy, are standard elements in the theory of
Wasserstein gradient flows. For the proof of  inequality
\eqref{internalenergysubdiff}, which characterizes the vector
$\nabla P(\rho)$ as the Wasserstein subdifferential of the
internal energy functional, we refer to \cite[§3.3.1]{AG} or to the general theory in
\cite[§10.4.3]{AGS}.

On the other hand, \eqref{internalenergysubdiff} can be used to
generalize the proof of Theorem \ref{th:GFEVI}.
Indeed, if $(\rho_t,v_t) $ solves \eqref{nonlinearpp}
according to our notion of solution, thanks to
\eqref{generalizedfisher} $\rho_t$ satisfies the identity
\eqref{internalenergysubdiff} for almost any $t$. From this
inequality, all the rest of the proof of Theorem \ref{th:GFEVI} can be carried out. Indeed, with the same
notation therein, we obtain the $\mathcal{L}^1$-a.e. $t\in (0,T)$
inequality
\begin{align*}
I_t:=\GG_{\eps,\alpha}(\bar\rho,\bar v)&-\GG_{\eps,\alpha}(\rho_t,
v_t)
+\int_{\mathbb{R}^d}(\bar v(x)- v_t(x))\bar\rho(x)\,dx\ge\\
&\int_{\mathbb{R}^d} \langle \rho_t(x)\nabla
P(\rho_t(x))-\rho_t(x)\nabla v(x),\mathcal{T}_t(x)-x \rangle \, dx
- C
\omega(W_2^2(\rho_t,\bar\rho)),
\end{align*}
for any $\bar\rho\in\MeasuresTwo\cap L^∞(\mathbb{R}^d)$ and any
$\bar v\in W^{1,2}(\mathbb{R}^d)$ if $\eps>0$ or
$\bar v=\mathcal{B}_{\alpha,d}\ast \bar\rho$ if $\eps=0$. This estimate
substitutes \eqref{I_t} in the proof of Theorem \ref{th:GFEVI}. The rest of the proofs is
completely analogous.
\end{proof}

\subsection*{Acknowledgements}

The authors would like to thank Paolo Acquistapace and Lucilla
Corrias for several discussions about this work. JAC acknowledges
support from the project MTM2011-27739-C04-02 DGI (Spain) and
2009-SGR-345 from AGAUR-Generalitat de Catalunya. JAC acknowledges
support from the Royal Society through a Wolfson Research Merit
Award. This work was partially supported by Engineering and
Physical Sciences Research Council grant number EP/K008404/1. SL
and EM has been partially supported by the INDAM-GNAMPA project
2011 "Measure solution of differential equations of
drift-diffusion, interactions and of Cahn-Hilliard type". EM has
been partially supported by a postdoctoral scholarship of the
Fondation Ma\-th\'e\-ma\-ti\-que Jacques Hadamard, he acknowledges
hospitality from Paris-Sud University.


\end{document}